\documentclass[12pt]{article}

\usepackage{amsmath, amsthm, amssymb}
\usepackage{geometry}
\usepackage{hyperref}
\usepackage{enumitem}

\newtheorem{theorem}{Theorem}[section]
\newtheorem{lemma}[theorem]{Lemma}
\newtheorem{corollary}[theorem]{Corollary}
\newtheorem{proposition}[theorem]{Proposition}
\theoremstyle{definition}
\newtheorem{definition}[theorem]{Definition}
\newtheorem{example}[theorem]{Example}
\newtheorem{problem}[theorem]{Open Problem}
\theoremstyle{remark}
\newtheorem{remark}[theorem]{Remark}

\newcommand{\PP}{\mathbb{P}}
\newcommand{\PPo}{\mathbb{P}_0}
\newcommand{\NN}{\mathbb{N}}
\newcommand{\bX}{\beta X}
\newcommand{\Xs}{X^{*}}
\newcommand{\fu}{\mathfrak{u}}
\newcommand{\fv}{\mathfrak{v}}
\newcommand{\fw}{\mathfrak{w}}

\newcommand{\fp}{\mathfrak{p}}
\newcommand{\fq}{\mathfrak{q}}
\newcommand{\fk}{\mathfrak{k}}
\newcommand{\fF}{\mathfrak{F}}
\newcommand{\fD}{\mathfrak{D}}
\newcommand{\calK}{\mathcal{K}}
\newcommand{\calP}{\mathcal{P}}

\newcommand{\Del}{\Delta}

\newcommand{\sing}{\mathrm{sing}}
\DeclareMathOperator{\card}{\#}
\newcommand{\pf}{parfilter}
\newcommand{\upf}{parultrafilter}
\newcommand{\Upf}{Parultrafilter}

\renewcommand{\hat}{\widehat} %%I dislike tiny hats.
\title{Maximal Filters in the Lattice of Partitions of an Infinite Set}
\author{David V.\ Feldman \\
  {\small University of New Hampshire, Durham} \\[4pt]
  Alexander Wilce \\
  {\small Susquehanna University}}
\date{August 2, 2026}

\begin{document}

\maketitle

\begin{abstract}
We study maximal (proper) filters in the complete lattice of partitions of an infinite set~$X$.  In the language of uniform spaces, these are precisely the atoms of the lattice of zero-dimensional uniformities on $X$, introduced by Pelant and Reiterman and studied further by Pelant,
Reiterman, R\"odl and Simon. The first half of this paper recovers, sharpens, and extends their classification in purely partition-theoretic terms.  {Call a maximal filter 
of partitions {\em type I} if it does not contain all finite partitions, and {\em type II} if it does.} 

Type I filters are induced, in an essentially unique way, by ultrafilters on families of pairwise disjoint
doubletons.  

{A type II filter} determines a non-principal ultrafilter on $X$, the \emph{heart}; the heart determines the filter precisely when {the former} is minimal in the Rudin--Keisler order.  Each member $F$ of {a type II filter} 
 gives rise to a closed \emph{fiber} in $\Xs = \bX \setminus X$, {consisting of} the set of ultrafilters agreeing with the heart on $F$. We prove a trichotomy describing the topology of arbitrary fibers.  We then show that the fibers do not encode the filter: fibers do not form a semilattice under intersection, the closure of an infinite discrete set of ultrafilters of a single Rudin--Keisler type (a \emph{sparse} set, in our terminology)  need not be a fiber, and a partition incompatible with a member of the filter may have a {fiber strictly larger than that member.}
%strictly larger fiber than that member. 
A representation that does succeed is nevertheless available in another category: the 
 
{type-II filters} with heart $\fu$ correspond to
the maximal proper substructures of the ultrapower $X^X\!/\fu$ of the full structure on $X$.  The topological representation problem remains open. 
\end{abstract}

% ------------------------------------------------------------------
\section{Introduction}
% ------------------------------------------------------------------

If $X$ is a non-empty set, denote by $\PP(X)$ the set of partitions of $X$ by non-empty subsets.  We say that a partition $E \in \PP(X)$ \textbf{refines}
another partition $F$ (or that $F$ is a \textbf{coarsening} of~$E$), writing $E \leq F$, iff every cell of $F$ is the union of cells of $E$.  Ordered by refinement, $\PP(X)$ is a complete lattice with minimal element $\Del = \bigl\{\{x\} \mid x \in X\bigr\}$ and maximal element $\{X\}$. In this paper we describe the maximal proper filters on this lattice.

If $L$ is any bounded lattice, a \textbf{(proper) filter} in $L$ is a non-empty proper subset $\fF \subseteq L$ such that
$a, b \in \fF \Rightarrow a \wedge b \in \fF$ and
$a \geq b \in \fF \Rightarrow a \in \fF$.
By Zorn's Lemma, proper filters may be extended to maximal proper filters. Maximal proper filters in $\PP(X)$ we call \upf{}s, and filters in $\PP(X)$ generally we call \pf{}s.  (Filters {and ultrafilters} of subsets of $X$, of closed sets, and so on, remain simply \emph{filters} {and \emph{ultrafilters}}.)
The \textbf{principal} \pf{} generated by $F \in \PP(X)$ is
$\{E \mid E \geq F\}$; this is maximal iff $F$ is an atom of~$\PP(X)$.
The atoms of $\PP(X)$ are the partitions containing a
single $2$-element cell and only singletons otherwise.  We are mainly interested in non-principal \upf{}s.

It is useful to note that a filter $\fF$ in a bounded lattice $L$ is maximal iff for every $p \in L$,
\[
  p \notin \fF \;\Longrightarrow\; p \wedge q = 0_L
  \quad\text{for some } q \in \fF,
\]
where $0_L$ is the least element of $L$.  (Given the displayed condition, any proper filter strictly containing $\fF$ contains some $p \notin \fF$ together with a $q \in \fF$ satisfying $p \wedge q = 0_L$, hence contains $0_L$; conversely if $\fF$ is maximal and $p \notin \fF$, the collection
$\{x \mid x \geq p \wedge q \text{ for some } q \in \fF\}$ is a filter properly containing $\fF$, hence improper, whence $p \wedge q = 0_L$ for some $q \in \fF$.)

In the special case $L = \calP(X)$ (the power set of~$X$), maximal filters are the usual \emph{ultrafilters on~$X$}.  We denote ultrafilters on sets by small Gothic letters $\fu$, $\fv$, etc., 
{and \pf{}s by capital gothic letters.  
We write $\beta X$ for the set of all ultrafilters 
on $X$ and $X^{\ast}$ for the set of non-principal 
ultrafilters.}  Subsets of $X$ are denoted by small Roman letters $a, b, \ldots$, while large Roman letters $E, F, \ldots$ and $A, B, \ldots$ denote partitions and subsets of partitions of~$X$.  We use the axiom of choice freely and without comment.

If $X$ is infinite, the collection $\PPo(X)$ of partitions of $X$ having only finitely many cells is a proper \pf.  Every \upf{} extending $\PPo(X)$ is non-principal because 
{any given atom can be split by 2-cell partitions.}
%many 2-cell partitions will split any given atom.

We call an \upf{} that does not extend $\PPo(X)$ \textbf{type~I} and
those that do, \textbf{type~II}.  ({Thus,} principal \upf{}s fall into type~I.)

This subject turns out to have had a previous life in the theory of uniform spaces, a fact we learned only after the results below were found.  A uniformity on $X$ {(understood here as a set of uniform covers)} is \emph{zero-dimensional} if it has a basis of covers 
that are partitions.  Since a partition star-refines itself and the meet of two partition covers is their common refinement, the map sending a proper \pf{} to the uniformity it generates 
is an order isomorphism between proper \pf{}s and non-discrete zero-dimensional uniformities on~$X$.  Under this dictionary, \upf{}s are exactly the \emph{atoms} of the lattice of zero-dimensional uniformities in the sense of Pelant--Reiterman
\cite{pelant-reiterman} and Pelant--Reiterman--R\"odl--Simon
\cite{prrs1,prrs2}: the maximal  non-discrete members.  Our type-I \upf{}s
are their \emph{proximally non-discrete} atoms (type II being the proximally discrete case: a type-II \upf{} contains all finite partitions and so induces the discrete proximity ), and Theorem~\ref{thm:typeI} below is, in partition-lattice form, their one-to-one correspondence between such atoms and ultrafilters --- the \emph{ultrasum} construction of
\cite{pelant-reiterman} --- sharpened here by an essential-uniqueness remark.  The heart (Theorem~\ref{thm:heart}) appears as \cite[Proposition~2.2]{prrs1}.  Theorem~\ref{thm:minimal} is, for sets of non-measurable cardinality, equivalent by Booth's theorem \cite{booth} to their Proposition~2.3 (the filter uniformity of  $\fu$ is an atom iff
$\fu$ is selective); the Rudin--Keisler formulation given here requires no cardinality hypothesis.  The material from Section~\ref{sec:fibers} onward --- the fiber trichotomy, the negative results on the representation problem, and the ultrapower representation --- appears to be new, though the last was anticipated in a remark of the referee of
\cite{prrs2} (see Section~\ref{sec:ultrapower}).  A different partition-semilattice ultrafilter theory, oriented by  coarsening and with all blocks required {to be} infinite, was developed by Matet \cite{matet} and by Halbeisen--L\"owe \cite{halbeisen-loewe} in the wake of the dual Ramsey theorem; the classification questions answered here are posed, in dualized form, in the problem list of \cite{halbeisen-loewe}.  For the
continuation of the atoms program to quasi-uniformities see de
Jager--K\"unzi \cite{dejager-kunzi}.

Section~\ref{sec:typeI} describes the type-I \upf{}s completely: they are exactly the \pf{}s $\fD_\fu$ induced by ultrafilters $\fu$ on families of pairwise disjoint doubletons, and the inducing pair  is essentially unique.

Section~\ref{sec:typeII-RK} shows that every type-II \upf{} determines a certain non-principal ultrafilter on the underlying set~$X$, called its \textbf{heart}.  An \upf{} is determined by its heart if and only if the heart has minimal type in the Rudin--Keisler order.  Two combinatorial lemmas proved there --- a splitting lemma and a coloring lemma --- do most of the work in the rest of the paper.

Section~\ref{sec:fibers} studies, for a partition $F$ and a non-principal  ultrafilter $\fu$ on $X$, the closed \emph{fiber} $f^*(\fu) \subseteq \Xs$
of ultrafilters agreeing with $\fu$ on $F$, and proves a trichotomy: the fiber is a copy of the growth of a single cell, or a finite set computed by an ultraproduct of cells, or the closure of an infinite discrete set of Rudin--Keisler-least points --- this last case obtaining for every partition 
in 
a type-II \upf.  We also show that the closure of an infinite discrete set of ultrafilters of a single Rudin--Keisler type (a \emph{sparse} set, in our terminology) need \emph{not} be a fiber, and pose the intrinsic characterization of fibers as an open problem.

Section~\ref{sec:representation} concerns the natural hope that a type-II \upf{} $\fF$ with heart $\fu$ can be recovered from the family $\calK(\fF) = \{f^*(\fu) \mid F \in \fF\}$ of fibers of its members, so that type-II \upf{}s would correspond to suitable maximal filters of closed sets in $\Xs$.  We show that this fails, twice over: the identity $f^*(\fu) \cap g^*(\fu) = (f \wedge g)^*(\fu)$, on which the earlier
argument rested, is false (Proposition~\ref{prop:meet-vs-cap}); and, more seriously, there exist (in ZFC)  a type-II \upf{} $\fF$ with heart $\fu$, a 
{partition} %member 
$F \in \fF$, and a partition $G$ with $F \wedge G = \Del$ --- so that $G$ is incompatible with $\fF$ --- whose fiber satisfies $g^*(\fu) \supsetneq f^*(\fu)$ (Theorem~\ref{thm:absorption-fails}).  Thus
comparison of fibers cannot detect membership in $\fF$.

Section~\ref{sec:ultrapower} gives %% the 
{a logical} representation {of type-II filters} that does work. Regarding $X$ as a first-order structure with a name for every finitary operation 
{on $X$ [?]} , we show that the type-II \upf{}s with heart $\fu$ correspond bijectively to the maximal proper substructures of the ultrapower $X^X\!/\fu$; the \pf{} $\langle\hat\fu\rangle$  corresponds to the diagonal copy of $X$, so that Theorem~\ref{thm:minimal} takes the form: the diagonal is a maximal proper substructure of $X^X\!/\fu$ if and only if $\fu$ is minimal in the Rudin--Keisler order.  What \emph{topological} data over $\Xs$ determines a type-II \upf{} we leave open.

% ------------------------------------------------------------------
\section{Type-I \Upf{}s}
\label{sec:typeI}
% ------------------------------------------------------------------

Henceforth $X$ is an infinite set.  If $E$ is a partition of $X$ and $a \subseteq X$, we say that $a$ is a \textbf{section} (resp. \textbf{sub-section}) of $E$ iff $a$ intersects each cell of $E$ in exactly
(resp.\ at most) one point.  Thus $E \wedge F = \Del$ iff every cell of {$F$ is a subsection of $E$ 
(and vice versa).} 
%$E$ is a sub-section of $F$.

Let $D$ be a non-empty collection of pairwise disjoint $2$-element subsets of~$X$.  To each subset $A \subseteq D$ we associate a partition
\[
  D_A \;=\; A \;\cup\; \bigl\{\{x\} \mid x \notin \textstyle\bigcup A\bigr\},
\]
and to each partition $F$ of $X$ the set
\[
  \Sigma_F \;=\; \bigl\{\, d \in D \mid d \text{ is contained in a cell of } F \,\bigr\}
\]
of doubletons that $F$ fails to split.  Since $F \geq D_A$ iff
$A \subseteq \Sigma_F$, and since filters are upward closed, an ultrafilter
$\fu$ on the set $D$ yields {the collection 
of partitions} 
\[
  \fD_\fu \;=\;
  \bigl\{\, F \in \PP(X) \mid \exists\, A \in \fu,\; F \geq D_A \,\bigr\}
  \;=\;
  \bigl\{\, F \in \PP(X) \mid \Sigma_F \in \fu \,\bigr\}.
\]
 Notice that {$A \subseteq B \subseteq D$ implies $D_{A} \leq D_{B}$; 
also }
$D_A \wedge D_B = D_{A \cap B}$ for all $A, B \subseteq D$. Thus, the collection $\{D_A \mid A \in \fu\}$ is closed under finite meets and 
does not contain $\Del$; hence $\fD_\fu$ is a proper \pf. Note that 
$\fD_\fu$ is principal iff $\fu$ is principal (in which case $\fD_\fu$ is the principal \pf{} of the atom $D_{\{d\}}$, $d$ being the doubleton generating $\fu$).
{Finally, since $D_{A}$ refines itself, 
$D_{A} \in \fD_{\fu}$ for all $A \in \fu$, and 
conversely, if $D_{A} \geq D_{B}$ for some 
$B \in \fu$, then $A \supseteq B$ and thus 
$A \in \fu$.}

\begin{theorem}[cf.\ {\cite[Section~2.1]{prrs1}}]
\label{thm:typeI}
If $\fu$ is an ultrafilter on a non-empty collection $D$ of pairwise  disjoint doubletons, then $\fD_\fu$ is a type-I \upf.
Every type-I \upf{} is of this form.
\end{theorem}

\begin{proof}
Let $\fu$ be an ultrafilter on $D$.  
%% For maximality we verify the criterion of Section~1. 
Suppose $F \notin \fD_\fu$, i.e.\
$\Sigma_F \notin \fu$.  Then $B := D \setminus \Sigma_F \in \fu$, and 
$F \wedge D_B = \Del$: each doubleton in $B$ is split by $F$, hence is a
sub-section of $F$, and the singleton cells of $D_B$ are trivially
sub-sections.  Thus $\fD_\fu$ is maximal.  To see that $\fD_\fu$ is of
type~I, choose a section $a$ of $D$ and note that $\{a, a^c\} \wedge D_A = \Del$ for every $A \subseteq D$, so the finite partition $\{a, a^c\}$ does not belong to $\fD_\fu$.

Conversely, suppose $\fF$ is a type-I \upf.  If a \pf{}
contains every two-element partition $\{a, a^c\}$ of $X$, then it contains
all finite partitions.  Thus there exist $a \subseteq X$ and $E \in \fF$ such
that $E \wedge \{a, a^c\} = \Del$.  Then no cell of $E$ has more than two
elements.  Let $D \subseteq E$ be the set of $E$'s non-singleton cells.

{Now, for any partition $F$ of $X$, 
let $\phi(F) = (F \wedge E ) \cap D$. That is, $\phi(F)$ is the set of doubletons of $E$ contained in a cell of $F$. If $F \in \fF$, then as 
$F \cap E \not = \Delta$, such doubletons exist, 
and so $\phi(F) \not = \emptyset$. It is easy to check 
that $\phi(F_1 \wedge F_2) = \phi(F_1) \cap \phi(F_2)$, 
so the image of $\fF$ under $\phi$ is the base for a filter, $\fu$, on $D$. We claim that $\fu$ is an 
ultrafilter on $D$. For any $A \subseteq D$ and $F \in \fF$, we have $D_{A} \wedge F = \Delta$ iff $A \cap \phi(F) = \emptyset$. {Now suppose $A \subseteq D$
with $A \notin \fu$.  Then $D_A \notin \fF$: for $\phi(D_A) = A$ (the
doubletons of $E$ contained in a cell of $D_A$ are exactly those of $A$),
so $D_A \in \fF$ would put $A$ in the base of $\fu$.  By the maximality
criterion of Section~1 there is then some $F \in \fF$ with
$D_A \wedge F = \Delta$, i.e., $A \cap \phi(F) = \emptyset$.} Similarly, if 
$A^{c} = D \setminus A \not \in \fu$, there is some 
$F' \in \fF$ with $D_{A^{c}} \wedge F' = \Delta$. 
But now all the doubletons of $D$ contained 
in cells of $F$ lie in $A^{c}$, while those of 
$F'$ lie in $A$; accordingly, $F \wedge F' \wedge E = \Delta$, a contradiction.}
{Thus $\fu$ is an ultrafilter.  Finally, for every $F \in \fF$ we
have $\phi(F) \in \fu$ and $\phi(F) \subseteq \Sigma_F$, so
$\Sigma_F \in \fu$ and $F \in \fD_\fu$.  Hence
$\fF \subseteq \fD_\fu$, and since $\fF$ is maximal and $\fD_\fu$ is
proper, $\fF = \fD_\fu$.} \end{proof}

Note that the proof of the converse required no non-principality hypothesis: if $\fF$ is principal, the set $D$ above is finite, $\fu$ is principal, and the representation still obtains.

\begin{remark}[Essential uniqueness]
\label{rem:uniqueness}
{The representing pair $(D, \fu)$ is unique up to restriction. Suppose $\fD_\fu = \fD_{\fu'}$ for 
ultrafilters $\fu, \fu'$ on sets $D, D'$ of doubletons. 
Then $D \cap D'$ is non-empty, else 
$\Delta = D_{D} \wedge D_{D'} \in \fD_{\fu} = \fD_{\fu'}$, a contradiction. 
Let $A \subseteq D$ with $A \in \fu$. Then 
$D_{A} \in \fD_{\fu} = \fD_{\fu'}$, so 
for some $B \subseteq D'$ with $B \in \fu'$, 
we have $D_{A} \geq D_{B}$. But then every 
doubleton in $B$ is also a doubleton in $A$, 
i.e., $B \subseteq A$, whence, $A \cap D' \in \fu'$. 
Similarly, every $B \in \fu'$ has 
$B \cap D \in \fu$. Thus, $\fu$ and $\fu'$ have 
the same trace on $D \cap D'$. }

\end{remark}

Call a partition $F \in \PP(X)$ \textbf{bounded} iff
$\sup_{a \in F} \card(a) < \infty$.  Theorem~\ref{thm:typeI} shows that
every type-I \upf{} is generated by partitions all of
whose cells have at most two elements.  By contrast, a type-II \upf{}
contains no bounded partition at all: a partition $F$ is bounded iff there
exists a finite partition $E$ with $E \wedge F = \Del$ (color the points of
each cell of $F$ injectively with $n = \sup_{a \in F}\card(a)$ colors, and
let $E$ be the partition into color classes; conversely if $E$ has $k$
cells and $E \wedge F = \Del$, each cell of $F$ is a sub-section of $E$ and
so has at most $k$ points). A type-II \upf{} contains $E$, {and so, cannot contain $F$.} 
%hence not $F$.  
Section~\ref{sec:typeII-RK} sharpens this observation
(Corollary~\ref{cor:unbounded}).

% ------------------------------------------------------------------
\section{Type-II \Upf{}s and the Rudin--Keisler Order}
\label{sec:typeII-RK}
% ------------------------------------------------------------------

To every subset $a \subseteq X$ we associate a partition
\[
  \hat{a} \;=\; \{a\} \;\cup\; \bigl\{\{x\} \mid x {~\in X \setminus a}\}
%\notin a\bigr\}.
\]
{We adopt the convention that $\hat a$ is defined for every
$a \subseteq X$, with $\hat a = \Del$ precisely when $a$ has at most one
element.  With this convention the identity
$\hat{a} \wedge \hat{b} = \widehat{a \cap b}$ holds unconditionally ---
resolving a query of the second author: for disjoint $a, b$ both sides are
$\Del$.  (This is also the convention of the machine formalization.)}  If $\fu$ is any
filter on $X$, let $\hat{\fu} = \{\hat{a} \mid a \in \fu\}$.  Since
$\hat{\fu}$ is closed under finite meets, the \pf{}
$\langle \hat{\fu} \rangle$ generated by $\hat{\fu}$ consists of coarsenings
of elements of $\hat{\fu}$.  As long as $\fu$ is non-principal, $\hat{\fu}$
does not contain $\Del$, and therefore $\langle \hat{\fu} \rangle$ is proper.

If $\fu$ is an ultrafilter and $E$ is a finite partition of $X$, then $\fu$
must contain some cell $a \in E$.  Thus, as $E \geq \hat{a}$, we have
$\PPo(X) \subseteq \langle\hat\fu\rangle$, and any
\upf{}   extending $\hat{\fu}$ is of type~II.

We will  establish that  
%%%, conversely, every type-II \upf{} arises as such an extension for a uniquenon-principal ultrafilter $\fu$ on~$X$.

If $\fF$ is a \pf, we define the \textbf{heart} of $\fF$,
written $h(\fF)$, to be the set of all $a \subseteq X$ such that
$\hat{a} \in \fF$.  Since the map $a \mapsto \hat{a}$ is meet-preserving
and order-preserving,
$h(\fF)$ is a filter on $X$.  If $b \in h(\langle\hat{\fu}\rangle)$, then
$\hat{b} \geq \hat{a}$ for some $a \in \fu$, whence $b \supseteq a$, whence
$b \in \fu$.  Thus $h(\langle\hat{\fu}\rangle) = \fu$.

\begin{theorem}[cf.\ {\cite[Proposition~2.2]{prrs1}}]
\label{thm:heart}
If $\fF$ is a type-II \upf, then $h(\fF)$ is a non-principal
ultrafilter on $X$.
\end{theorem}

\begin{proof}
Let $a \subseteq X$.  Then $\{a, a^c\} \in \fF$.  Suppose that neither
$\hat{a}$ nor $\widehat{a^c}$ belongs to $\fF$.  Since $\fF$ is maximal,
there exist partitions $F_1, F_2 \in \fF$ with
$\hat{a} \wedge F_1 = \widehat{a^c} \wedge F_2 = \Del$.
But then $\{a, a^c\} \wedge (F_1 \wedge F_2) = \Del$: each cell of
$F_1 \wedge F_2$ meets $a$ at most once and meets $a^c$ at most once, so its
intersections with the two cells of $\{a,a^c\}$ are singletons or empty.
This contradicts the properness of $\fF$.  
{
Finally, since $\widehat{\{x\}} = \Delta \not \in \fF$ 
for all $x \in X$, $h(\fF)$ is non-principal.}
%Finally, $h(\fF)$ is non-principal: were it principal, it would contain a singleton $\{x\}$,and then $\widehat{\{x\}} = \Del \in \fF$, absurd.
\end{proof}

Two simple lemmas now do a surprising amount of work.  For a partition $F$, write $\sing(F)$ for the union of the singleton cells of $F$ and, for $n \geq 1$, $\sing_n(F)$ for the union of the cells of $F$ with at most $n$ elements (so $\sing_1 = \sing$).

\begin{lemma}[Splitting Lemma]
\label{lem:splitting}
For any partition $F$, {$\hat{\sing(F)} \wedge F = \Delta$.} 
%with $s = \sing(F)$ we have
%$\hat{s} \wedge F = \Del$ (provided $s \neq \emptyset$; if $s = \emptyset$ the conclusion is vacuous in the applications below).
\end{lemma}

\begin{proof} 
The cells of $\hat s \wedge F$ are the sets $s \cap c$ ($c$ a cell of $F$) together with singletons drawn from $s^c$.  If $c$ is a singleton cell then
$c \subseteq s$ and $s \cap c = c$ is a singleton; if $c$ is not a singleton then $c \cap s = \emptyset$.  So every cell of $\hat s \wedge F$ is a singleton.
\end{proof}

\begin{lemma}[Coloring Lemma]
\label{lem:coloring}
Let $\fF$ be a type-II \upf{} with heart $\fu$, and let $F \in \fF$.  Then
for every $n \geq 1$, $\sing_n(F) \notin \fu$.
\end{lemma}

\begin{proof}
First the case $n = 1$.  If $s = \sing(F) \in \fu$ then
$\hat s \in \fF$, and by Lemma~\ref{lem:splitting},
$\hat s \wedge F = \Del$ with both meetands in $\fF$, contradicting
properness.

Now suppose $\sing_n(F) \in \fu$ for some $n \geq 2$.  Color the points of
each cell of $F$ of size at most $n$ injectively with the colors
$1, \ldots, n$, and let $E$ be the finite partition of $X$ into the $n$
color classes together with $X \setminus \sing_n(F)$ (omitted if empty).
Since $\fF$ is of type II, $E \in \fF$, so $G := F \wedge E \in \fF$.  {Every cell of $G$ 
meets $\sing_{n}(F)$ in a singleton} 
%On $\sing_n(F)$, every cell of $G$ is a singleton 
(a cell of size $\leq n$
meets each color class at most once).  Hence
$\sing(G) \supseteq \sing_n(F) \in \fu$, so $\sing(G) \in \fu$,
contradicting the case $n = 1$ applied to $G$.
\end{proof}

\begin{definition}
Let $\fu$ be a non-principal ultrafilter on $X$ and $F$ a partition.  Say $F$ is \textbf{$\fu$-bounded} if $\sing_n(F) \in \fu$ for some $n$, and
\textbf{$\fu$-unbounded} otherwise.
\end{definition}

\begin{corollary}
\label{cor:unbounded}
Every member of a type-II \upf{} is $\fu$-unbounded, where $\fu$ is the
heart.
\end{corollary}

The next proposition determines exactly which partitions {
are compatible with a given heart.} 
%can accompany a prescribed heart.

\begin{proposition}[Compatibility]
\label{prop:compat}
Let $\fu$ be a non-principal ultrafilter on $X$ and $F$ a partition.  Then
$F$ belongs to some type-II \upf{} with heart $\fu$ if and only if no
member of $\fu$ is a sub-section of $F$.
\end{proposition}

%{---- To here ----}

\begin{proof}
For $a \subseteq X$ non-empty, $\hat{a} \wedge F = \Del$ iff every cell of
$F$ meets $a$ at most once, i.e.\ iff $a$ is a sub-section of $F$.  
{
In this case, no proper parfilter contains both $\hat{a}$ and $F$. Thus, if $F$ has a subsection $a \in \fu$, no parfilter extending $\hat{\fu}$, and so, 
no type-II \upf{} with heart $\fu$, contains $F$. 
}

Conversely, 
if no member of $\fu$ is a sub-section of $F$, then $\hat a \wedge F \neq \Del$ for every $a \in \fu$. 
{Hence the \pf{} generated by $\hat\fu \cup \{F\}$ is proper (its
general element coarsens some $\hat a \wedge F \neq \Del$), and by Zorn's
Lemma it extends to an \upf.} {Thus, $F$ belongs to a \upf{} $\fF$ that 
contains $\hat{\fu}$ (and is thus type~II), 
and $h(\fF) \supseteq \fu$, whence $h(\fF) = \fu$. 
}

\end{proof}

In general it is possible to extend {a given} $\langle\hat{\fu}\rangle$ in many incompatible ways to an \upf.

\begin{example}
\label{ex:fubini}
Let $X = \NN \times \NN$, let $\fp$ and $\fq$ be non-principal ultrafilters
on $\NN$, and let $\fu = \fp \otimes \fq$ be their Fubini product:
$a \in \fu$ iff $\{m \mid \{n \mid (m,n) \in a\} \in \fq\} \in \fp$.  Let
$V = \bigl\{\{m\} \times \NN \mid m \in \NN\bigr\}$ (columns) and
$R = \bigl\{\NN \times \{n\} \mid n \in \NN\bigr\}$ (rows).  A sub-section
$a$ of $V$ meets each column at most once, so each set
$\{n \mid (m,n) \in a\}$ has at most one element and is not in $\fq$;
hence $a \notin \fu$.  A sub-section $a$ of $R$ meets each row at most
once, so $a = \{(g(n), n) \mid n \in \operatorname{dom} g\}$ for a partial
function $g$ {from row indices to column indices}; for each $m$ the trace of $a$ on the $m$-th column is $\{(m,n) \mid n \in g^{-1}(m)\}$, and the sets $g^{-1}(m)$ are pairwise disjoint, at most one of which can lie in $\fq$, so $\{m \mid \{n \mid (m,n) \in a\} \in \fq\}$ has at most one element and is not in $\fp$; hence $a \notin \fu$.  By Proposition~\ref{prop:compat}, both $V$ and $H$ belong to type-II \upf{}s
with heart $\fu$; but $V \wedge R = \Del$, so no \upf{} contains both.
(One readily checks that the pushforwards of $\fu$ along $V$ and $R$, in the sense introduced next, are $\fp$ and $\fq$ respectively.)
\end{example}

Recall that for any set $X$, $\bX$ denotes the Stone space of $\calP(X)$, i.e., the set of all ultrafilters on~$X$, with principal ultrafilters identified with points of~$X$.  {Any mapping $f : X \rightarrow Y$ 
yields $\beta f : \beta X \rightarrow \beta Y$ given 
by}
%If $f\colon X \to Y$ is any function, we obtain {a mapping [function...]} $\beta f\colon \bX \to \beta Y$ by
\[
  \beta f(\fu) \;=\; \bigl\{b \subseteq Y \mid f^{-1}(b) \in \fu\bigr\}.
\]
If $f$ is a surjection, so is $\beta f$: for $\fv \in \beta Y$, $(\beta f)^{-1}(\fv)$ is the set of ultrafilters extending the filter
generated by $\{f^{-1}(b) \mid b \in \fv\}$.
In particular, if $F$ is a partition of $X$ and $f\colon X \to F$ is the
canonical surjection (sending each $x$ to the cell of $F$ containing it),
then
\[
  \beta f(\fu) \;=\; \bigl\{B \subseteq F \mid \textstyle\bigcup B \in \fu\bigr\}.
\]

If $\fv = \beta f(\fu)$ for some function $f$, we write $\fv \leq \fu$.
This defines the \textbf{Rudin--Keisler pre-order} on ultrafilters.  (To
avoid set-theoretic scruples it suffices, throughout, to consider
ultrafilters on quotients of $X$ --- equivalently, on sets of cardinality
at most $\card(X)$; the types below a fixed ultrafilter then form a set.)  Two
ultrafilters $\fu$ and $\fv$ belong to the same \textbf{type} iff
$\fu \leq \fv$ and $\fv \leq \fu$; the induced partial order on types  is the
\textbf{Rudin--Keisler order}.  Every ultrafilter below $\fu$ arises, up to
a bijection of index sets, as $\beta f(\fu)$ for the canonical surjection
of a partition $F$ of $X$: given $g \colon X \to Y$, replace $Y$ by the partition of $X$ into the fibers of $g$.

It is a standard result \cite{comfort-negrepontis} that $\fv = \beta f(\fu)$ is of the same type as $\fu$ if and only if there exists $a \in \fu$ on
which $f$ is injective.  When $f\colon X \to F$ is the canonical surjection of a partition, $f$ is injective on $a$ iff $a$ is a sub-section of $F$;
and since $\fu$ is upward closed and any sub-section extends to a section,
this is equivalent to the assertion that $\fu$ contains a section of $F$.

The type of the principal ultrafilters is the least element of the Rudin--Keisler order; an ultrafilter dominated by a principal one is itself
principal.  A non-principal ultrafilter is \textbf{minimal} iff it is
equivalent to every non-principal ultrafilter lying beneath it in the
Rudin--Keisler order.

\begin{theorem}[cf.\ {\cite[Proposition~2.3]{prrs1}}]
\label{thm:minimal}
A non-principal ultrafilter $\fu$ on $X$ is of minimal type if and only if
$\langle\hat{\fu}\rangle$ is an \upf.
\end{theorem}

\begin{proof}
First observe that for any partition $F$ with canonical surjection $f$:
\begin{enumerate}[label=(\roman*)]
\item $F \in \langle\hat\fu\rangle$ iff $F \geq \hat a$ for some
  $a \in \fu$ iff some cell of $F$ belongs to $\fu$ iff $\beta f(\fu)$ is
  principal;
\item for $a \in \fu$, $\hat a \wedge F = \Del$ iff $a$ is a sub-section of
  $F$ iff $f$ is injective on $a$.
\end{enumerate}
Since every element of $\langle\hat\fu\rangle$ coarsens some $\hat a$
($a \in \fu$), and since $G \geq \hat a$ and $F \wedge G = \Del$ force
$F \wedge \hat a = \Del$, the maximality criterion of Section~1 reads:
$\langle\hat\fu\rangle$ is an \upf{} iff for every partition $F$, either
$\beta f(\fu)$ is principal or some $a \in \fu$ is a sub-section of $F$.
By the standard result quoted above, the latter alternative says exactly
that $\beta f(\fu)$ has the same type as $\fu$.  
{Thus, $\langle \hat{\fu} \rangle$ is an 
\upf{} iff every ultrafilter below $\fu$ in the Rudin-Keisler order is either principal or equivalent 
to $\fu$ --- that is, iff $\fu$ is minimal.}  

\end{proof}

\begin{remark}
\label{rem:selective}
For ultrafilters on a countable set, minimality is equivalent to
selectivity (the Ramsey property); see Booth~\cite{booth} and
\cite{comfort-negrepontis}.  Selective ultrafilters exist under the
Continuum Hypothesis \cite{booth, comfort-negrepontis}, while Kunen showed
that after adding $\aleph_2$ random reals to a model of CH there are none
(see \cite{jech}); so the existence of hearts that determine their
\upf{}s is, for countable $X$, independent of ZFC.  For general $X$: if a
minimal $\fu$ is countably incomplete, then some countable partition of
$X$ has no cell in $\fu$, its pushforward is a non-principal ultrafilter
on a countable set lying below $\fu$, and minimality makes $\fu$
isomorphic to it; so every countably incomplete minimal ultrafilter is
isomorphic to a selective ultrafilter on $\NN$.  The countably complete
case is bound up with measurable cardinals, and we do not pursue it here.
In uniform language, Theorem~\ref{thm:minimal} says that the filter
uniformity of $\fu$ is an atom of the lattice of \emph{zero-dimensional}
uniformities iff $\fu$ is {Rudin-Keisler} minimal; compare
\cite[Proposition~2.3]{prrs1}, where the same uniformity is shown to be
an atom of the lattice of \emph{all} uniformities, for $X$ of
non-measurable cardinality, iff $\fu$ is selective in the partition sense
used there.
\end{remark}

\begin{remark}
\label{rem:baumgartner}
We record here, without proof, two assertions that we believe deserve
further study: if a type-II \upf{} $\fF$ is recovered from
$\langle\widehat{h(\fF)}\rangle$ by adjoining a single partition, then
$h(\fF)$ dominates a minimal ultrafilter with no type strictly between;
and, as James Baumgartner observed to us, this situation occurs when
$h(\fF)$ is the product of a minimal ultrafilter with itself.  Under CH, the situation contemplated here does
occur: for any prescribed ultrafilter $\mathcal G$ on $\omega$,
Pelant--Reiterman--R\"odl--Simon construct hearts above $\mathcal G$ in
the Rudin--Keisler order admitting exactly $s$ type-II \upf{}s
($1 \leq s < \omega$), each obtained from $\langle\hat\fu\rangle$ by
adjoining a single partition \cite[Theorem~4.1]{prrs2}; a heart admitting
exactly one, which is \emph{not} so obtainable
\cite[Theorem~4.2]{prrs2}; and a heart admitting $2^{\mathfrak c}$ of
them \cite[Theorem~4.3]{prrs2}.  (Their counts concern atoms of the full
lattice of uniformities, but since the atoms in question are
zero-dimensional and every \upf, viewed as a uniformity, is refined by
some atom, the counts transfer to \upf{}s.)
\end{remark}

% ------------------------------------------------------------------
\section{Fibers}
\label{sec:fibers}
% ------------------------------------------------------------------

For a function $f$ and a set $a$ we write $f[a]$ for the image
$\{f(x) \mid x \in a\}$.  It will be convenient to identify a partition $F$ of $X$ with the associated
surjection $f\colon X \to F$.  Thus we write $f \in \fF$ to mean $F \in
\fF$, and $f \wedge g$ for the surjection $X \to F \wedge G$; the surjection
associated to a partition $F, G, \ldots$ will always be denoted $f, g,
\ldots$, and vice versa.  For $a \subseteq X$ we write
$\overline{a} = \{\fw \in \bX \mid a \in \fw\}$ for the basic clopen subset of $\bX$ determined by $a$, and $a^* = \overline{a} \cap \Xs$.

Since $X$ is open in $\bX$, the set $\Xs = \bX \setminus X$ of non-principal
ultrafilters on $X$ (the \textbf{growth} of $X$) is compact. Theorem~\ref{thm:heart} asserts that $h(\cdot)$ maps the set of type-II \upf{}s onto $\Xs$ (surjectivity following from
Proposition~\ref{prop:compat} with $F = \{X\}$, or simply by extending
$\langle\hat\fu\rangle$), and
Theorem~\ref{thm:minimal} asserts that $h$ has a cross-section over the set
of minimal types.

For a partition $F$ and $\fu \in \Xs$, define the \textbf{fiber} {of $\fu$ over $F$ to be the 
set of non-principal ultrafilters on $X$ agreeing 
with $\fu$ on $F$, that is, the set}
\[
  f^*(\fu) \;:=\; (\beta f)^{-1}\bigl(\beta f(\fu)\bigr) \cap \Xs,
\]
%the set of non-principal ultrafilters on $X$ that agree with $\fu$ on the partition $F$.  
The fiber is closed, contains $\fu$, and satisfies
$\fv \in f^*(\fu) \Rightarrow f^*(\fv) = f^*(\fu)$ 
  Throughout this section, {we} fix $F$ {and} $\fu$, put $\fp = \beta f(\fu)$ and $K = f^*(\fu)$, and
let $\fk$ denote the filter on $X$ generated by
$\{f^{-1}(B) \mid B \in \fp\}$.%\\ 

Thus, $(\beta f)^{-1}(\fp)$ is the set of ultrafilters extending $\fk$. When $\fp$ is non-principal, every such extension is non-principal 
{(since if $\fv$ is a principal ultrafilter on $X$, $\beta f(\fv)$ is principal on $F$).} 
%(each cell $c$ satisfies $X \setminus c \in \fk$), 
%so that 
{Thus,} $K$ is all of $(\beta f)^{-1}(\fp)$ and $\fk = \bigcap K$ ({as any} filter {is} the intersection of the ultrafilters extending it).

\begin{lemma}[Sections]
\label{lem:sections}
Suppose $\fp$ is non-principal and let $j$ be a section of $F$.  Then the
filter $\fk[j]$ generated by $\fk \cup \{j\}$ is an ultrafilter; it is the
unique element of $K$ containing $j$ (so $\fk[j]$ is isolated in $K$, being
the sole point of the relatively clopen set $\overline{\jmath} \cap K$);
and it has the same type as $\fp$, which is the least type occurring
in~$K$.
\end{lemma}

\begin{proof} 
{
The filter generated by $\fk$ and $j$ has a base 
consisting of sets $f^{-1}(B) \cap j = \bigcup B \cap j$ for $B \in \fp$. Letting $s_{j} : F \rightarrow j$ be the bijection taking $c \in F$ to the unique element of $c \cap j$, we have $\bigcup B \cap j = 
s_{j}(B)$. Thus, this filter base is in fact 
the image under $s_{j}$ of $\fp$, and hence, an 
ultrafilter on $j$. But a filter containing an ultrafilter on one of its members is already an ultrafilter.}
%A filter containing $\fk$ and $j$ contains $f^{-1}(B) \cap j {= \bigcup B \cap j}$ for every$B \in \fp$, and these sets are non-empty since $j$ meets every cell {in $B$}. For $b \subseteq j$ we have $b = j \cap f^{-1}(f[b])$because $j$ meets eachcell exactly once; as $\fp$ is an ultrafilter, either $f[b] \in \fp$,putting $b$ in our filter, or $F \setminus f[b] \in \fp$, putting$j \setminus b$ in it.  So the trace on $j$ is an ultrafilter on $j$, and afilter containing an ultrafilter on one of its members is an ultrafilter.
Clearly $\fk[j] \in K$.  Any $\fw \in K$ with $j \in \fw$ extends $\fk \cup \{j\}$, {and hence 
equals $\fk[j]$.} 
%hence extends the ultrafilter $\fk[j]$, hence equals it.
Since $f$ is injective on $j \in \fk[j]$ and
$\beta f(\fk[j]) = \fp$, the ultrafilter $\fk[j]$ has the type of $\fp$, {by the Comfort-Negrepontis 
result \cite{comfort-negrepontis} quoted earlier.}
%\cite{comfort-negrepontis}; 
and $\fp = \beta f(\fw)$ for every
$\fw \in K$, so this type is least in $K$.
\end{proof}

{Let} %Write
$T = T_F(\fu)$ {be the set $\{\fk[j] \mid j \text{ a section of } F\}$ of ultrafilters arising 
from sections of $F$}.

\begin{lemma}[Density]
\label{lem:density}
If $\fp$ is non-principal, $T$ is dense in $K$.
\end{lemma}

\begin{proof}
Let $\fw \in K$ and $a \in \fw$. {It will be enough to} %we 
find a point of $T$ in $\overline a \cap K$.  Let $C_a = f[a]$, the set of cells meeting $a$ Since $a \subseteq {\bigcup C_{a} =} f^{-1}(C_a)$, we {have $f^{-1}(C_a) \in \fw$ 
and hence,} %get 
$C_a \in \beta f(\fw) = \fp$.  
{
Select a point of $c \cap a$ for every $c \in C_{a}$, and extend this to a section $j$ of $F$.}
%Choose a section $j$ of $F$ that selects on each cell $c \in C_a$, a point of $c \cap a$, and on other cells selects any point.  
Then $f^{-1}(C_a) \cap j \subseteq a$. {As} %and
$f^{-1}(C_a) \cap j \in \fk[j]$, {it follows 
that} %so
$a \in \fk[j]$, i.e.\ $\fk[j] \in \overline a \cap K$.
\end{proof}

\begin{proposition}[Three descriptions of $T$]
\label{prop:threefaces}
Suppose $\fp$ is non-principal.  Then $T$ is exactly the set of isolated
points of $K$, and exactly the set of points of $K$ of least (i.e.\
$\fp$'s) type.  Moreover $\fk[j] = \fk[j']$ iff
$\{c \in F \mid j \cap c = j' \cap c\} \in \fp$, so that $j \mapsto \fk[j]$
induces a bijection from the ultraproduct
$\prod_{c \in F} c \,/\, \fp$ onto $T$.
\end{proposition}

\begin{proof}
Points of $T$ are isolated by Lemma~\ref{lem:sections}
{and $T$ is dense by Lemma~ref{lem:density}, and thus, contains every isolated point.}  
%Conversely if$\fw$ is isolated in $K$, say $\{\fw\} = \overline a \cap K$ with $a \in \fw$, then by Lemma~\ref{lem:density} some point of $T$ lies in $\overline a \cap K$, so $\fw \in T$.  
Points of $T$ have least type by Lemma~\ref{lem:sections}; conversely if $\fw \in K$ has the type of $\fp$,
then since $\beta f(\fw) = \fp$, the standard criterion
\cite{comfort-negrepontis} provides $a \in \fw$ on which $f$ is injective,
i.e., a sub-section of $F$; extending $a$ to a section $j$ we get
$j \in \fw$ by upward closure, so $\fw = \fk[j] \in T$ by uniqueness.

For the last statement, if $A := \{c {\in F} \mid j \cap c = j' \cap c\} \in \fp$,
then $b := f^{-1}(A) \cap j = f^{-1}(A) \cap j'$ lies in both $\fk[j]$ and
$\fk[j']$; the filter generated by $\fk \cup \{b\}$ already has an
ultrafilter trace on $b$ (the argument of Lemma~\ref{lem:sections} applies to the partial section $b$, which meets every cell in $A$), so it has a
unique ultrafilter extension, whence $\fk[j] = \fk[j']$.  Conversely if $\fk[j] = \fk[j'] =: \fw$, then $j \cap j' \in \fw$, so
$f[j \cap j'] = {\{ c \in F \mid j \cap j' \cap c \not = \emptyset\} \in f(\fw) = \fp.}$ 
{But as $j$ and $j'$ are sections, 
$j \cap j' \cap c \not = \emptyset$ iff 
$j \cap c = j' \cap c$.}
%$\{c {\in F} \mid j \cap c = j' \cap c\} \in \beta f(\fw) = \fp$. 
\end{proof}

%[Added definitions of $F$ $\fu$-bounded and $\fu$-unbounded here:]
{
[Comment: For the following, we should spell out somewhere 
(here or in section 3) that if $\fp = \beta f(\fu)$, 
then $F$ is $\fu$-bounded iff for some $n \in \NN$, the set of cells of $F$ having $n$ or fewer elements belongs to $\fp$ (Let $F_n$ be set of cells of size $\leq n$; then $f^{-1}(F_n) = \bigcup F_n = \sing_n(F)$.) ]}

Note that $F$ is $\fu$-bounded iff, with $\fp = \beta f(\fu)$, the set of
cells of $F$ with at most $n$ elements lies in $\fp$ for some $n$: the
union of those cells is $\sing_n(F)$.

\begin{theorem}[Fiber Trichotomy]
\label{thm:trichotomy}
Let $\fu \in \Xs$ and let $F$ be a partition of $X$, with
$\fp = \beta f(\fu)$ and $K = f^*(\fu)$.  Exactly one of the following
holds.
\begin{enumerate}[label=\emph{(\alph*)}]
\item Some cell $c$ of $F$ belongs to $\fu$.  Then $c$ is infinite,
  $K = c^*$, and $K$ has no isolated points.
\item No cell of $F$ belongs to $\fu$, and $F$ is $\fu$-bounded.  Then
  $K = T$ is finite; if $\fp$-almost all cells of $F$ have exactly $k$ elements, then $\card(K) = k$.  In particular
  $K = \{\fu\}$ iff $\sing(F) \in \fu$.
\item No cell of $F$ belongs to $\fu$, and $F$ is $\fu$-unbounded.  Then
  $T$ is infinite and discrete, its points all of type $\fp$, and
  $K = \overline{T}$; the isolated points of $K$ are exactly $T$.
\end{enumerate}
\end{theorem}

\begin{proof}
(a) Here $\fp$ is principal {at some cell} $c \in F$, {and as $c \in \fu$ with $\fu$ non-principal, $c$ is infinite.} 
%; $c \in \fu$ with $\fu$ non-principalforces $c$ infinite.  
An ultrafilter $\fv$ agrees with $\fu$ on $F$ iff
$c \in \fv$, so $K = c^*$.  For any $\fv \in c^*$ and $a \in \fv$, the set
$a \cap c$ lies in $\fv$ and is infinite; splitting it into two infinite 
halves produces two distinct points of $c^*$ inside $\overline a$, so no point is isolated.

For (b) and (c), $\fp$ is non-principal and the lemmas above apply. 

(b) Choose $n$ and $A_0 \in \fp$ with every cell in $A_0$ of size at most
$n$; label the points of each such cell injectively with
$1, \ldots, n$.  Each section $j$ {of $F$ [?] then}  determines a label function
$A_0 \to \{1, \ldots, n\}$. Since $\fp$ is an ultrafilter, exactly one 
label value $i$ occurs on a $\fp$-large set of cells, and two sections with
the same such value agree (as selections) on a $\fp$-large set of cells,
whence $\fk[j] = \fk[j']$ by Proposition~\ref{prop:threefaces}.  Thus
$\card(T) \leq n$. Being finite, $T$ is closed, and by
Lemma~\ref{lem:density}, $K = \overline{T} = T$.  If the sizes {of the cells in $A_o$} are exactly
$k$ $\fp$-almost everywhere, the $k$ constant-label sections pairwise disagree on a $\fp$-large set, giving $\card(K) = k$ exactly.  Finally
$\sing(F) \in \fu$ iff the cells {in $A_o$} are singletons $\fp$-almost everywhere
iff $k = 1$ iff $K = \{\fu\}$ ({noting} $\fu \in K$ always).

(c) $\fu$-unboundedness says that for each $m$, the set $N_m$ of cells with at least $m$ elements lies in $\fp$.  Fix $m$ and choose sections
$j_1, \ldots, j_m$ pairwise disagreeing on every cell of $N_m$ (possible since such cells have at least $m$ points).  For $i \neq l$, the agreement
set of $j_i, j_l$ misses $N_m \in \fp$, so by
Proposition~\ref{prop:threefaces} the $\fk[j_i]$ are pairwise distinct. 
Hence $\card(T) \geq m$ for all $m$: $T$ is infinite.  Discreteness, common type, $K = \overline T$ (density plus closedness of $K$), and the
identification of the isolated points are
Lemmas~\ref{lem:sections},~\ref{lem:density} and
Proposition~\ref{prop:threefaces}. 
% (The three cases are mutually exclusive by their hypotheses.)
\end{proof}

\begin{corollary}
\label{cor:singleton-fiber}
For any partition $F$ and $\fu \in \Xs$:\; $f^*(\fu) = \{\fu\}$ if and
only if $\sing(F) \in \fu$.
\end{corollary}

\begin{proof}
In case (a), $K = c^*$ {is infinite}, 
so $K \not = \{\fu\}$, {and $c$ is} infinite, 
{so} $\sing(F) \cap c = \emptyset$,
exclud{ing} %es 
$\sing(F) \in \fu$. In case (b), this is part of the statement.  In case (c), $K$ is infinite and $\sing(F) = \sing_1(F) \notin \fu$ by
$\fu$-unboundedness.
\end{proof}

\begin{remark}
By the classical theorem of Frayne, Morel and Scott
\cite{fms} (see also \cite{chang-keisler}), an ultraproduct over a
countably incomplete ultrafilter is finite or of cardinality at least
$2^{\aleph_0}$.  Hence in case (c), if $\fp$ is countably incomplete ---
as it must be when $X$ is countable --- then
$\card(T) \geq 2^{\aleph_0}$.
\end{remark}

\begin{definition} 
A set $K \subseteq \Xs$ is \textbf{sparse} iff $K$ is the closure of an
infinite discrete set of ultrafilters of a common type.  (We avoid the word ``rare,'' which has an established, different meaning for ultrafilters \cite[1.6]{prrs2}.)
\end{definition}

\begin{corollary}
\label{cor:members-sparse}
Suppose no cell of $F$ lies in $\fu$ and $F$ is $\fu$-unbounded.  Then {the fiber} $f^*(\fu)$ is sparse.  In particular, if $\fF$ is a type-II \upf{} with heart $\fu$ and $F \in \fF$, then either $F \in \langle\hat\fu\rangle$ (equivalently, some cell of $F$ lies in $\fu$) and $f^*(\fu) = c^*$ for
that cell, or $f^*(\fu)$ is sparse.
\end{corollary}

\begin{proof}
The first statement is Theorem~\ref{thm:trichotomy}(c) together with the
definition of sparseness.  For the second: the equivalence
$F \in \langle\hat\fu\rangle \iff$ some cell of $F$ lies in $\fu$ was
observed in the proof of Theorem~\ref{thm:minimal}; when it fails, $F$ is
$\fu$-unbounded by Corollary~\ref{cor:unbounded}, and the first statement
applies.  (We take no position on whether the growths $c^*$ arising in the
first alternative are themselves sparse.)
\end{proof}

\begin{proposition}[Sparse sets need not be fibers]
\label{prop:sparse-not-fiber}
Let $X = \NN \times \NN$ and let $\fp$ be a non-principal ultrafilter on
$\NN$.  Let $V$ be the partition into columns, $K_V = (\beta v)^{-1}(\fp)$
the corresponding fiber, and let
$S = \{\fk_V[j_n] \mid n \in \NN\}$, where $j_n = \NN \times \{n\}$ is the
$n$-th row (a section of $V$).  Then $\overline S$ is sparse, but
$\overline S$ is not of the form $g^*(\fv)$ for any partition $G$ of $X$
and any $\fv \in \Xs$.  In fact $\overline S \subsetneq K_V$.
\end{proposition}

\begin{proof} { By Lemma \ref{lem:sections}, each 
$\fk[j_k] \in S$ is the unique ultrafilter in $K$ containing $j_k$. As distinct rows are disjoint, 
no two belong to a common ultrafilter, and thus, 
the $\fk[j_k]$ are distinct. So $S$ is countably infinite. Lemma \ref{lem:sections} also tells us 
that each $\fk[j]$ has type $\fp$, and is an isolated point of $K_{V}$. It is therefore also isolated in $\overline{S}$. Since $K_{V}$ is closed, $\overline{S} \subseteq K_{V}$.  Since $\overline{S}$ is the closure of a discrete set, its isolated points are exactly those of $S$. }
{So $\overline S$ is sparse, with isolated-point set the countably
infinite $S$.}

{For the strict containment, let $j_{\Delta} = \{(n,n) \mid n\}$ be
the diagonal section of $V$.  Since the ultrafilters in $K_V$ are
non-principal and $j_k \cap j_{\Delta}$ is a singleton, no member of $K_V$
contains both $j_k$ and $j_\Delta$; thus
$\fk_V[j_{\Delta}] \notin S$.  Being isolated in $K_V$
(Lemma~\ref{lem:sections}), $\fk_V[j_{\Delta}]$ then does not belong to
$\overline{S}$ either: its isolating neighborhood meets $\overline S$ only
in points of $S$.  Thus $\overline{S} \subsetneq K_{V}$; in particular
$\overline S$ is not $v^{\ast}(\fv)$ for any $\fv \in K_{V}$.} 
%Can this be leveraged to simplify the rest of the 
%proof?]} 
%To obtain the rest of the proof, it would suffice to know that (i) fibres over incomparable partitions are incomparable, and (ii) then deal with the more limited case in which $F \leq G$ (in which case fibre over $\fq = \beta \phi (\fp) \in \beta G$ contains fibre over $\fp \in \beta F$.) ]}

Now suppose $\overline S = g^*(\fv)$ for some partition $G$ and
$\fv \in \Xs$, and let $\fq = \beta g(\fv)$.  If $\fq$ is principal at a
cell $c$ of $G$, then $c$ is infinite and $g^*(\fv) = c^*$ has no isolated
points (Theorem~\ref{thm:trichotomy}(a)), contradicting
$S \neq \emptyset$.  So $\fq$ is non-principal. Since $X$ is countable,
$G$ is countable and {thus} $\fq$ is countably incomplete.  By Theorem~\ref{thm:trichotomy}(b,c) the isolated points of $g^*(\fv)$ form
the set $T_G(\fv)$, which is in bijection with the ultraproduct 
$\prod_{c \in G} c\,/\,\fq$; by Frayne--Morel--Scott \cite{fms, chang-keisler} this ultraproduct is finite or of cardinality at least 
$2^{\aleph_0}$.  In the finite case, $g^*(\fv) = T_G(\fv)$ is finite (Theorem~\ref{thm:trichotomy}(b)), contradicting the infinitude of
$\overline S$; in the other case the isolated points of $g^*(\fv)$ have cardinality at least $2^{\aleph_0} \neq \aleph_0 = \card(S)$.  Either way, 
we contradict the fact that the isolated points of $\overline S$ are exactly the countably infinite set $S$. 
\end{proof}

\begin{problem}
\label{prob:fibers}
Characterize intrinsically (topologically, or in terms of the
Rudin--Keisler structure of their points) the closed subsets of $\Xs$ of the form $f^*(\fu)$ ($F$ a partition of $X$, $\fu \in \Xs$).  By Theorem~\ref{thm:trichotomy} and
Proposition~\ref{prop:sparse-not-fiber}, in the $\fu$-unbounded case
sparseness is necessary but not sufficient.
\end{problem}

% ------------------------------------------------------------------
\section{The Representation Problem}
\label{sec:representation}
% ------------------------------------------------------------------

Let $\fF$ be a type-II \upf{} with heart $\fu$, and put
\[
  \calK(\fF) \;=\; \bigl\{f^*(\fu) \mid f \  \ \in \fF\bigr\}.
\]
%{be the set of the fibers over $\fu$. It is easy to see that if $G \leq F$, then $f^{\ast}(\fu) \subseteq g^{\ast}(\fu)$. Thus, $f \mapsto f^{\ast}(\fu)$ gives an orer-reversing mapping from $\fF$ to $\calK(\fF)$, the latter ordered by inclusion. It might be hoped that the order-structure of $\calK(\fF)$ might determine $\fu$. However, as we shall show,  }

One might hope that $\calK(\fF)$ is a maximal filter in a suitable
semilattice of closed subsets of $\Xs$, and that the assignment
$\fF \mapsto \calK(\fF)$ identifies the type-II \upf{}s with heart
$\fu$ with maximal filters of sparse closed sets converging properly to
$\fu$.  The purpose of this section is to show that this hope fails, to
record what remains true, and to pose what we believe is the correct open
problem.

We begin with what is true.

\begin{proposition}
\label{prop:positive}
Let $\fF$ be a type-II \upf{} with heart $\fu$.  Then:
\begin{enumerate}[label=\emph{(\arabic*)}]
  \item $(f \wedge g)^*(\fu) \subseteq f^*(\fu) \cap g^*(\fu)$ for all
    $F, G$; consequently $\calK(\fF)$ is downward directed under
    inclusion.
  \item For all $f \in \fF$, $f^*(\fu) \neq \{\fu\}$.
  \item $\bigcap \calK(\fF) = \{\fu\}$: the family $\calK(\fF)$ converges
    properly to the heart, which is therefore recoverable from it.
\end{enumerate}
\end{proposition}

\begin{proof}
(1) If $\fv$ agrees with $\fu$ on $F \wedge G$ it agrees with $\fu$ on
every coarsening of $F \wedge G$, in particular on $F$ and on $G$
(pushforwards factor through the canonical surjections
$F \wedge G \to F$, $F \wedge G \to G$).  Directedness follows since
$f \wedge g \in \fF$ whenever $f, g \in \fF$.

(2) By Lemma~\ref{lem:coloring}, $\sing(F) \notin \fu$; apply
Corollary~\ref{cor:singleton-fiber}.

(3) For $a \in \fu$ we have $\hat a \in \fF$, and the fiber of $\hat a$ at
$\fu$ is computed directly: any ultrafilter containing $a$ has pushforward
along $\hat a$ equal to the principal ultrafilter at the cell $a$, while
an ultrafilter omitting $a$ has a pushforward omitting $\{a\}$; since
$a \in \fu$, an ultrafilter $\fv \in \Xs$ agrees with $\fu$
on $\hat a$ iff $a \in \fv$.  So
$\hat a^*(\fu) = a^*$, and
$\bigcap \calK(\fF) \subseteq \bigcap_{a \in \fu} a^* = \{\fv \in \Xs \mid
\fu \subseteq \fv\} = \{\fu\}$.  Conversely, $\fu$ lies in every fiber $f^*(\fu)$.
\end{proof}

The hope rests on the assertion that the containment in Proposition~\ref{prop:positive}(1) is an equality --- that fibers at $\fu$ form a semilattice under intersection, with
$f^*(\fu) \cap g^*(\fu) = (f \wedge g)^*(\fu)$.  This is false:\footnote{Again, I'd urge revising this paragraph to avoid reference to the previous draft.}

\begin{proposition}
\label{prop:meet-vs-cap}
Let $\fq$ be any non-principal ultrafilter on $\NN$, let $X = \NN \times \NN$, and let $V$ and $R$ be the partitions of $X$ into
columns and rows.  The family
$\{(B \times B) \setminus \Del_X \mid B \in \fq\}$, where
$\Del_X = \{(n,n) \mid n\}$ is the diagonal, has the finite intersection
property, Let $\fw$ be any ultrafilter on $X$ extending {this family}, and let
$\sigma\fw$ be its image under the transposition
$\sigma(m,n) = (n,m)$.  Then $\fw \neq \sigma\fw$, yet
\[
  \sigma\fw \;\in\; v^*(\fw) \cap r^*(\fw) 
  \qquad\text{while}\qquad
  (v \wedge r)^*(\fw) = \Del^{\,*}(\fw) = \{\fw\}.
\]
Hence $f^*(\fu) \cap g^*(\fu) \supsetneq (f \wedge g)^*(\fu)$ in general.
\end{proposition}

\begin{proof}
{For the} finite intersection property, {note that} $(B_1 \times B_1) \cap \cdots \cap
(B_k \times B_k) \setminus \Del_X = (B \times B) \setminus \Del_X$ with
$B = \bigcap B_i \in \fq$ {non-empty, since $\fq$ is non-principal}. %infinite, hence non-empty.  
{Let $\fw$ be an ultrafilter extending this family}. Then $\fw$  is non-principal: {if 
it were} principal at $(m,n)$, {we would have} %require
$m, n \in B$ for every $B \in \fq$, while $\bigcap \fq = \emptyset$.  
%For any such $\fw$, 
{We also have} 
$v^{-1}(B) = B \times \NN \supseteq (B \times B) \setminus \Del_X$ for
{every} $B \in \fq$, so $\beta v(\fw) \supseteq \fq$, whence
$\beta v(\fw) = \fq$; symmetrically $\beta r(\fw) = \fq$.  Since
$v \circ \sigma = r$ and $r \circ \sigma = v$, we get
$\beta v(\sigma\fw) = \beta r(\fw) = \fq = \beta v(\fw)$ and likewise for
$r$; so the (non-principal) $\sigma\fw$ agrees with $\fw$ on both $V$ and
$H$.  But $\fw \neq \sigma\fw$: with $u = \{(m,n) \mid m < n\}$ we have
$\Del_X^{\,c} = u \sqcup \sigma(u) \in \fw$, so exactly one of
$u, \sigma(u)$ lies in $\fw$; if $u \in \fw$ then, as $\sigma$ is an
involution, $u \in \sigma\fw$ would mean $\sigma(u) \in \fw$, {which is} impossible; and symmetrically if $\sigma(u) \in \fw$.  Finally
$V \wedge R = \Del$, and $\Del^*(\fw) = \{\fw\}$ by
Corollary~\ref{cor:singleton-fiber} (every cell of $\Del$ is a
singleton).
\end{proof}

Proposition~\ref{prop:meet-vs-cap} already wrecks the program:
the purported maximality of $\calK(\fF)$ among filters of closed sets, and
the separation principle ``$e \wedge g = \Del$ implies
$e^*(\fu) \cap g^*(\fu) = \{\fu\}$'' used to reconstruct $\fF$ from
$\calK$, both fail.  (Take $\fu = \fw$ as above and any type-II \upf{}
$\fF \ni V$ with heart $\fw$ --- one exists by
Proposition~\ref{prop:compat} provided $\fw$ contains no sub-section of
$V$, which can be arranged by adding to the generating family the
complements of all sub-sections of $V$: the finite intersection property
survives, since $(B \times B) \setminus \Del_X$ has infinitely many
points in each of infinitely many columns, while $k$ sub-sections of $V$
meet each column in at most $k$ points.)

{---- Skipping next remark for now (looks OK on a quick read, but should check reference....)} 

\begin{remark}
\label{rem:conditionR}
The configuration of Proposition~\ref{prop:meet-vs-cap} is tied to a
property of ultrafilters isolated in \cite[Section~3]{prrs1}: say that
$\fu$ satisfies \emph{condition (R)} if any two self-maps of $X$ with the
same pushforward along $\fu$ agree on a member of $\fu$.  The ultrafilter
$\fw$ above violates (R) as badly as possible: the two coordinate
projections (composed with a fixed bijection $\NN \to \NN \times \{0\}
\subseteq X$, say) have the same pushforward $\fq$, yet agree only on the
diagonal, which is not in $\fw$.  Pelant--Reiterman--R\"odl--Simon show
(R) to be strictly weaker than selectivity and equivalent to
$0$-proximal fineness of the associated atom \cite[Section~3]{prrs1}.
Whether Problem~\ref{prob:injective} below has a positive answer for
hearts satisfying (R) --- for which the pathology above cannot arise in
this form --- we do not know.
\end{remark}

One might
still hope that the \emph{family} of fibers, with its inclusion order,  remembers $\fF$.  The following containment criterion shows what fiber
inclusion means combinatorially.

\begin{lemma}[Containment criterion]
\label{lem:containment}
Let $F, G$ be partitions and $\fu \in \Xs$, with $\fp = \beta f(\fu)$
non-principal and $\fq = \beta g(\fu)$.  Then
$f^*(\fu) \subseteq g^*(\fu)$ if and only if for every $B \in \fq$ there
is $A \in \fp$ with $f^{-1}(A) \subseteq g^{-1}(B)$.
\end{lemma}

\begin{proof}
{($\Leftarrow$)  Here the hypothesis is the displayed combinatorial
condition, and the containment of fibers is what we prove.  Let
$\fv \in f^*(\fu)$; then $\fv$ extends $\fk_F$.  Given $B \in \fq$, the
hypothesis provides $A \in \fp$ with $f^{-1}(A) \subseteq g^{-1}(B)$;
since $f^{-1}(A) \in \fk_F \subseteq \fv$, also $g^{-1}(B) \in \fv$.
Thus $\beta g(\fv) \supseteq \fq$, whence
$\beta g(\fv) = \fq = \beta g(\fu)$ and $\fv \in g^*(\fu)$.} 
($\Rightarrow$)  As noted before Lemma~\ref{lem:sections}, when $\fp$ is
non-principal $\fk_F$ is the intersection of the ultrafilters in
$f^*(\fu)$.  If $f^*(\fu) \subseteq g^*(\fu)$, then every such ultrafilter
contains $g^{-1}(B)$ for each $B \in \fq$; hence
$g^{-1}(B) \in \fk_F$, which is the displayed condition.
\end{proof}

If membership in $\fF$ were readable from fibers, one would expect at
least this much \emph{absorption}: if $f \in \fF$ and
$g^*(\fu) \supseteq f^*(\fu)$, then $g \in \fF$.  Our main negative result
is that absorption fails, and fails as badly as possible: the larger fiber
can belong to a partition \emph{incompatible} with $\fF$.

\begin{theorem}
\label{thm:absorption-fails}
There exist an infinite set $X$, a type-II \upf{} $\fF$ on $\PP(X)$ with
heart $\fu$, and partitions $F \in \fF$ and $G \notin \fF$ such that
$F \wedge G = \Del$ (so that $G$ is excluded from \emph{every} \upf{}
containing $F$) and yet
\[
  g^*(\fu) \;\supsetneq\; f^*(\fu).
\]
\end{theorem}

\begin{proof}
Fix a countably infinite set $Y$ and a non-principal ultrafilter $\fq$ on
$Y$.  {Letting $\Upsilon$ denote the set of nonempty finite subsets of $Y$, let}
% and {let $Y$ be the coproduct of the non-empty finite subsets of $X$. Concretely, 
%Let $\Upsilon$ be the collection of all nonempty 
%finite subsets of $Y$, and set 
%\[
$  X \;=\; \bigl\{(T, b) \mid b \in T \in \Upsilon \bigr\},$  
  %\in [Y]^{<\omega} \setminus \{\emptyset\},\;
  %b \in T \bigr\},
%\]
{the coproduct of $\Upsilon$. 
The surjections $f : X \to \Upsilon$ and $g : X \rightarrow Y$ be 
the surjections }
$f(T,b) = T$ and $g(T,b) = b$ {induce corresponding partitions} %and let 
$F$ and $G$, %be the {corresponding} partitions of $X$ into the fibers of $f$ and of $g$. %respectively. 
{and clearly, $F \wedge G = \Delta$. 
We will produce an ultrafilter $\fp$ on $\Upsilon$ 
and an ultrafilter $\fu$ on $X$ such that 
the fibre of $f$ through $\fu$ is properly 
contained in the fibre of $g$ through $\fu$.} 
%We identify the index set of $F$ with $\Upsilon$. 
%$I = [Y]^{<\omega} \setminus \{\emptyset\}$.  
%The cell of $F$ over $T \in \Upsilon$ I$ is $\{T\} \times T$, of cardinality $\card(T)$, and $g$ isinjective on {each such cell}; hence every cell of $F \wedge G$ is a singleton:$F \wedge G = \Del$.

For $B \subseteq Y$ and $m \geq 1$ put
\[
  S_B = \{T \in \Upsilon \mid T \subseteq B\}, \qquad
  L_m = \{T \in \Upsilon \mid \card(T) \geq m\}.
\]
The family $\{S_B \mid B \in \fq\} \cup \{L_m \mid m \geq 1\}$ has the finite intersection property: $S_{B_1} \cap \cdots \cap S_{B_k} \cap L_m$
contains every $m$-element subset of $B_1 \cap \cdots \cap B_k \in \fq$, which is infinite.  Let $\fp$ be an ultrafilter on $\Upsilon$ extending this
family; since the $L_m$ decrease to $\emptyset$, $\fp$ is non-principal.

%Next let $\fu$ be an ultrafilter on $X$ extending
{Now consider the family} 
\[
  \{f^{-1}(A) \mid A \in \fp\} \;\cup\;
  \{X \setminus a \mid a \text{ a sub-section of } F\}.
\]
This %family 
{also} has the finite intersection property: given $A \in \fp$ and
sub-sections $a_1, \ldots, a_k$ of $F$, choose $T \in A \cap L_{k+1}
\in \fp$; the cell over $T$ has more than $k$ elements while each $a_i$
meets it at most once, so
$f^{-1}(A) \setminus (a_1 \cup \cdots \cup a_k) \neq \emptyset$.  Let $\fu$ be any ultrafilter on $X$ 
extending this family. Then
$\beta f(\fu) \supseteq \fp$, whence $\beta f(\fu) = \fp$; and $\fu$ is non-principal (since every singleton is a sub-section of $F$, so $\fu$
contains its complement).  By construction, no member of $\fu$ is a sub-section of $F$, so by Proposition~\ref{prop:compat} there
is a type-II \upf{} $\fF \ni F$ with heart $\fu$.  Since $F \wedge G = \Del$ and $\fF$ is proper, $G \notin \fF$.

It remains to compare the fibers.  First,
\[g^{-1}(B) = \{(T,b) \mid b \in B \cap T\} \supseteq
\{(T,b) \mid T \subseteq B,\ b \in T\} = f^{-1}(S_B)\] for every
$B \subseteq Y$.  For $B \in \fq$ we have $S_B \in \fp$, so
$f^{-1}(S_B) \in \fu$ and hence $B \in \beta g(\fu)$; thus $\beta g(\fu) = \fq$.  Now if $\fv \in f^*(\fu)$, then $\fv$ extends $\fk_F$, hence contains $f^{-1}(S_B)$, and so, {by the above, contains} %hence 
$g^{-1}(B)$, for every
$B \in \fq$; so $\beta g(\fv) = \fq = \beta g(\fu)$ and (as $\fv \in
\Xs$) $\fv \in g^*(\fu)$.  Thus $g^*(\fu) \supseteq f^*(\fu)$. 

Finally, the containment is strict.  Fix $B_0 \in \fq$ with $B_0 \neq Y$ and a point $b_0 \in Y \setminus B_0$.  The family
$\{g^{-1}(B) \mid B \in \fq\} \cup \{X \setminus f^{-1}(S_{B_0})\}$ has
the finite intersection property. 

{(The sets $g^{-1}(B)$ are closed under finite
intersections, so a finite subfamily reduces to a single $g^{-1}(B)$
together with the new set.)} {For this, it is enough to show that 
each set $g^{-1}(B)$ 
%note that the sets $g^{-1}(B)$ are closed under finite intersections, so it suffices to {show they} intersect %one of them with the
intersects $X \setminus f^{-1}(S_{B_0})$. %new set. 
%For $B \in \fq$ choose 
Choose $b \in B \setminus \{b_0\}$ and let} $T = \{b_0, b\}$: then $(T, b) \in g^{-1}(B)$ while $T \not\subseteq B_0$, so $(T,b) \notin f^{-1}(S_{B_0})$.  
An ultrafilter $\fv$ extending
this family is non-principal (it contains $g^{-1}(Y \setminus \{b\})$ for every $b$, since $\fq$ is non-principal, {and the intersection of these sets is empty) and} satisfies $\beta g(\fv) = \fq$, so $\fv \in g^*(\fu)$; but
$f^{-1}(S_{B_0}) \notin \fv$, {while} $S_{B_0} \in \fp$, so $\beta f(\fv) \neq \fp$ and {thus} $\fv \notin f^*(\fu)$.
\end{proof}

{ Theorem~\ref{thm:absorption-fails} shows that  the filter of closed sets generated by $\calK(\fF)$ 
cannot detect membership in $\fF$: retaining the 
notation of the foregoing proof, the filter generated 
by $\calK(\fF)$ contains $g^{\ast}(\fu)$ for a 
partition $G$ incompatible with $\fF$. In Section 
\ref{sec:ultrapower}, we will provide an alternative 
representation of type-II parultrafilters with heart $\fu$ in terms of the maximal proper substructures of the ultrapower $X^{X}\!/\fu$. However, before 
turning to that, we raise two questions. }

\begin{problem}
\label{prob:injective}
Is the assignment $\fF \mapsto \calK(\fF)$ injective on type-II \upf{}s
with a common heart $\fu$?  Equivalently: if $g^*(\fu) = f^*(\fu)$ with
$f \in \fF$, must $g \in \fF$?  
The case where $\beta g(\fu)$ is principal is
trivial, since then a cell of $G$ lies in $\fu$ and
$g \in \langle\hat\fu\rangle \subseteq \fF$; the substance of the question lies in the non-principal case.
\end{problem}

\begin{problem}
\label{prob:representation}
{Find a \emph{topological} representation of 
type-II ultrafilters with a given heart $\fu$ (or of the  corresponding substructures of $X^{X}\!/\fu$), either in terms of maximal families of closed subsets  of $\Xs$, or in other intrinsically topological terms. }
%In particular, Section~\ref{sec:ultrapower} identifies the type-II \upf{}s with heart $\fu$ with the maximal proper substructures of the ultrapower $X^X\!/\fu$; it remains open whether these objects also admit a workable description as maximal families of closed subsets of $\Xs$, or in other intrinsically topological terms.  (Proposition~\ref{prop:compat} shows the partitions available to $\fF$ are those without a sub-section in $\fu$; the problem is to describe which maximal families of suchpartitions arise.)
\end{problem}

% ------------------------------------------------------------------
\section{Type-II \Upf{}s and Ultrapowers}
\label{sec:ultrapower}
% ------------------------------------------------------------------

The representation that Section~\ref{sec:representation} shows cannot
proceed through the fibers turns out to be available in model-theoretic terms.  
For $X = \omega$ this was pointed out by the referee of \cite{prrs2} (see the remark opening Section~4 there), with reference to Keisler's survey \cite{keisler}; we give a direct partition-lattice proof, valid for every infinite $X$.

Fix a non-principal ultrafilter $\fu$ on $X$, and let $M = X^X\!/\fu$ denote the set of {equivalence} classes $[g]$ of maps $g \colon X \to X$ under the
relation of agreement on a member of~$\fu$.  {We 
identify $X$ with the set of constant classes, and}
%We 
regard $X$ as the \emph{full structure}: a first-order structure with a name for every finitary operation $h \colon X^n \to X$ and {for} every {
$n$-ary relation $R \subseteq X^{n}$.}
% {$n$-ary [?]} relation on $X$.  
Then $M$ is %its 
the ultrapower {of this structure} modulo $\fu$, with the operations acting by $h^M([g_1], \ldots, [g_n]) = [h \circ (g_1, \ldots, g_n)]$ (well defined,
as agreement on members of $\fu$ is preserved).  By a
\emph{substructure} of $M$ we mean a non-empty subset closed under every
operation $h^M$. {Applying constant functions, we see that every substructure contains a copy of $X$.} 
%(apply constant functions(applying constant operations shows that every substructure contains every constant class.  
Since Skolem functions for the full structure are among its named operations, the Tarski--Vaught
test shows {that} every substructure of $M$ 
{is} %to be 
an elementary submodel; we shall not need this fact, but it explains the model-theoretic reading of Theorem~\ref{thm:ultrapower} below.

For $g \colon X \to X$ let $G_g$ denote the partition of $X$ into the
non-empty fibers of $g$. {Every partition $F$ arises as $G_g$ for some
$g$; in this case, we call}
%(compose the canonical surjection with a choice of representatives); we call such a 
$g$ a \emph{realization} of $F$.  Call a class $[g]$
\emph{$\fu$-injective} if $g$ is injective on some member of $\fu${, noting that this} 
depends only on the class.

\begin{lemma}
\label{lem:generators}
For $[g] \in M$ the following are equivalent:
\begin{enumerate}[label=\emph{(\roman*)}]
  \item $[g]$ is $\fu$-injective;
  \item the substructure generated by $[g]$ is all of $M$.
\end{enumerate}
Consequently, a substructure of $M$ is proper if and only if it contains
no $\fu$-injective class.
\end{lemma}

\begin{proof}
The substructure generated by $[g]$ is $\{[k \circ g] \mid k \colon X^{n}
\to X, ~{n \in \NN}\}$: indeed $h \circ (g, \ldots, g) = (h \circ \Delta_n) \circ g$,
where $\Delta_n$ is the diagonal map $x \mapsto (x, \ldots, x)$.  If $g$
is injective on $a \in \fu$, then given any $h \colon X \to X$ define
$k$ on $g[a]$ by $k(g(x)) = h(x)$ and arbitrarily elsewhere; then
$k \circ g = h$ on $a$, so $[h] = [k \circ g]$.  As every element of $M$
is the class of a self-map, this proves (ii).  Conversely, if (ii)
holds, then $[\mathrm{id}] = [k \circ g]$ for some $k$, so
$\mathrm{id} = k \circ g$ on some $a \in \fu$, forcing $g$ to be
injective on~$a$.
\end{proof}

\begin{lemma}
\label{lem:welldef}
\begin{enumerate}[label=\emph{(\arabic*)}]
  \item If {$[g] = [g']$,} 
  %$g = g'$ on a member of $\fu$, 
  then for every \pf{} $\fF$
    with $\hat\fu \subseteq \fF$:\; $G_g \in \fF$ iff $G_{g'} \in \fF$.
  \item If $G_g = G_{g'}$, then for every substructure $S$ of $M$,\;
    $[g] \in S$ iff $[g'] \in S$.
\end{enumerate}
\end{lemma}

\begin{proof}
(1) Say $g = g'$ on $a \in \fu$.  The cells of $G_g \wedge \hat a$ are
the non-empty fibers of $g{\restriction}a$ together with singletons drawn
from $a^c$; hence $G_g \wedge \hat a = G_{g'} \wedge \hat a$.  Since
$\hat a \in \fF$, we have $F \in \fF$ iff $F \wedge \hat a \in \fF$, for
every partition $F$ (upward closure one way, closure under meets the
other). Apply this to $F = G_g$ and to $F = G_{g'}$.

(2) If $g, g'$ have the same fibers, {let $\tau_o : \operatorname{im}(g) \rightarrow \operatorname{im}(g')$ be the bijection taking $y \in \operatorname{im}(g)$ to the unique element of $g'(g^{-1}(y))$)} 
%the rule $\tau_0(g(x)) := g'(x)$ well-defines a bijection from $\operatorname{im} g$ onto$\operatorname{im} g'$; 
Extending $\tau_0$ arbitrarily to
$\tau \colon X \to X$ gives $g' = \tau \circ g$, whence
{if $[g] \in S$,} $[g'] = \tau^M([g]) \in S$. 
%{and symmetrically for the reverse inclusion.} %The converse is symmetric.
\end{proof}

\begin{theorem}
\label{thm:ultrapower}
Let $\fu$ be a non-principal ultrafilter on the infinite set $X$.  The
assignments
\begin{align*}
  S(\fF) &\;=\; \bigl\{[g] \in M \mid G_g \in \fF\bigr\},\\
  \fF(S) &\;=\; \bigl\{F \in \PP(X) \mid [g] \in S
    \text{ for some (equivalently, every) realization $g$ of } F\bigr\}
\end{align*}
are mutually inverse bijections between the type-II \upf{}s $\fF$ with
heart $\fu$ and the maximal proper substructures $S$ of
$M = X^X\!/\fu$.
\end{theorem}

\begin{proof}
Both assignments are well defined by Lemma~\ref{lem:welldef} (part (1)
applies since $\fu$ is the heart of $\fF$; part (2) gives the
parenthetical equivalence).

\emph{$S(\fF)$ is a proper substructure.}  If $[g_1], \ldots, [g_n] \in
S(\fF)$ and $h \colon X^n \to X$, put $g = h \circ (g_1, \ldots, g_n)$;
two points identified by every $g_i$ are identified by $g$, so
$G_g \geq G_{g_1} \wedge \cdots \wedge G_{g_n} \in \fF$ and {thus} 
$[g] \in S(\fF)$.  %Properness: 
{To see that $S(\fF)$ is proper, note 
that} if {$S(\fF) = M$, then by Lemma~\ref{lem:generators}, some} $[g] \in S(\fF)$ 
{would be} %were $\fu$-injective, with $g$ 
injective on some $a \in \fu$. {But} then $a$ would be a sub-section of
$G_g$, so $\hat a \wedge G_g = \Del$ with both meetands in $\fF$, {which is impossible}. 
% --- impossible.  Now apply Lemma~\ref{lem:generators}.

\emph{$S(\fF)$ is maximal.}  Let $S'$ be a substructure with $S' \supsetneq S(\fF)$, and pick $[g] \in S' \setminus S(\fF)$.  Then $G_g \notin \fF$, so by maximality of $\fF$ there is {some} $E \in \fF$ with $E \wedge G_g = \Del$.  Fix a realization $e$ of $E$ (so $[e] \in S(\fF) \subseteq S'$) and a pairing bijection $p \colon X \times X \to X$ (one exists, $X$ being infinite).  Since $E \wedge G_g = \Del$, the map
$(e, g) {: X \to X \times X}$ is injective, hence so is $p \circ (e, g)$; and
$[p \circ (e, g)] = p^M([e], [g]) \in S'$.  By
Lemma~\ref{lem:generators}, $S' = M$.  Thus no proper substructure strictly contains $S(\fF)$.

\emph{$\fF(S)$ is a proper \pf{} extending $\hat\fu$.}  Upward closure: 
if $F \leq F'$ and $g$ realizes $F$, choose for each cell of $F'$ a
representative and let $\sigma$ send each value $g(x)$ to the
representative of the $F'$-cell of $x$ (well defined since $F$ refines
$F'$, and extended arbitrarily off $\operatorname{im} g$); then
$\sigma \circ g$ realizes $F'$ and
$[\sigma \circ g] = \sigma^M([g]) \in S$.  Meets: if $g_F, g_G$ realize
$F, G$, then $p \circ (g_F, g_G)$  realizes $F \wedge G$, and its class lies in $S$.  Properness: a realization of $\Del$ is injective, and $S$,
being proper, contains no $\fu$-injective class
(Lemma~\ref{lem:generators}).  Non-emptiness: constants realize 
$\{X\}$.  The heart: for $b \in \fu$, {any} %the
map fixing $X \setminus b$
pointwise and collapsing $b$ to a point of $b$ realizes $\hat b$ and is
$\fu$-almost-everywhere constant, so its class is a constant class and lies in $S$; hence $\hat b \in \fF(S)$.  For $b \notin \fu$, the same map
is $\fu$-almost-everywhere the identity, so $\hat b \in \fF(S)$ would
put $[\mathrm{id}]$ in $S$, {contradicting 
the properness of $S$.}  Thus $\{a \mid \hat a \in \fF(S)\} = \fu$, and since every finite partition has a cell in $\fu$, $\fF(S)$ contains
$\PPo(X)$.

\emph{$\fF(S)$ is maximal} (hence a type-II \upf{} with heart $\fu$, by the preceding paragraph).  Suppose $G \notin \fF(S)$, with realization
$g$, so $[g] \notin S$.  The substructure generated by $S \cup \{[g]\}$
consists of the classes $h^M([e_1], \ldots, [e_m], [g], \ldots, [g])$ with $[e_1], \ldots, [e_m] \in S$ (a single application of an operation suffices, since
compositions of operations are operations); by maximality of $S$ it is
improper, hence equals $M$ and contains $[\mathrm{id}]$:
\[
  \mathrm{id} \;=\; h \circ (e_1, \ldots, e_m, g, \ldots, g)
  \quad\text{on some } a \in \fu,
\]
with $[e_i] \in S$ (possibly $m = 0$).  On $a$, then, the point $x$ is
determined by the values $e_1(x), \ldots, e_m(x), g(x)$; that is, $a$ is
a sub-section of $E_1 \wedge \cdots \wedge E_m \wedge G$, where
$E_i := G_{e_i} \in \fF(S)$.  Consequently
$(\hat a \wedge E_1 \wedge \cdots \wedge E_m) \wedge G = \Del$, and the
left meetand belongs to $\fF(S)$.  By the maximality criterion of
Section~1, $\fF(S)$ is an \upf.

\emph{Mutual inverses.}  $\fF(S(\fF)) = \fF$ is immediate from the
definitions, and $S(\fF(S)) = S$ follows from
Lemma~\ref{lem:welldef}(2).
\end{proof}

\begin{corollary}
\label{cor:diagonal}
%The constant classes --- 
The canonical (diagonal) copy of $X$ in its
ultrapower % --- form 
{is} a maximal proper substructure of $X^X\!/\fu$ if and
only if $\fu$ is minimal in the Rudin--Keisler order.
\end{corollary}

\begin{proof}
The formula defining $S(\cdot)$ makes sense for any \pf{} extending
$\hat\fu$, and for $\langle\hat\fu\rangle$ it yields exactly the constant
classes: $G_g \in \langle\hat\fu\rangle$ iff some cell of $G_g$ lies in
$\fu$ (observation (i) in the proof of Theorem~\ref{thm:minimal}) iff
$g$ is constant on a member of $\fu$ iff $[g]$ is a constant class.  If
the constants form a maximal proper substructure $S$, then
$\fF(S) = \langle\hat\fu\rangle$ by the same computation, so
$\langle\hat\fu\rangle$ is an \upf{} by
Theorem~\ref{thm:ultrapower}, and $\fu$ is minimal by
Theorem~\ref{thm:minimal}.  Conversely, if $\fu$ is minimal, then
$\langle\hat\fu\rangle$ is an \upf{} and
$S(\langle\hat\fu\rangle)$ --- the constants --- is a maximal proper
substructure.
\end{proof}

\begin{remark}
{(i)} Under this correspondence the questions of
Section~\ref{sec:representation} take algebraic form:
Problem~\ref{prob:injective} asks whether the closed family $\calK(\fF)$ determines the substructure $S(\fF)$.  

{(ii)} 
Counting type-II
\upf{}s with a given heart is {equivalent to} counting maximal proper substructures of
the ultrapower. The constructions of \cite[Section~4]{prrs2} provide,
under CH, hearts %realizing any finite number, 
{arising from any finite number, and 
from $2^{\mathfrak c}$, distinct type-II 
parultrafilters.} 
%and $2^{\mathfrak c}$ realizations}
(cf.\ Remark~\ref{rem:baumgartner}). 
\end{remark}

\subsection*{Verification}
The results of this paper have been formalized and machine-checked in
Lean~4 over the Mathlib library, in a development produced with the
Aristotle theorem-proving system; it is archived at
\texttt{https://doi.org/10.5281/zenodo.21876810}.  Partitions are modelled as
\texttt{Setoid}s under Mathlib's refinement order, and the ultrapower of
Section~\ref{sec:ultrapower} as a space of germs.  Every theorem,
proposition, corollary and lemma above is verified from the standard
foundations (propositional extensionality, choice, quotient soundness),
with two classical inputs assumed as explicitly stated, unproved lemmas
--- the Rudin--Keisler criterion quoted in
Section~\ref{sec:typeII-RK} \cite{comfort-negrepontis}, and the
Frayne--Morel--Scott theorem \cite{fms}, the latter consumed in the
specialized form that the isolated-point set of a fiber in case (c) of
Theorem~\ref{thm:trichotomy} is uncountable when $X$ is countable ---
and with the following provisos.  The content of
Lemmas~\ref{lem:sections}--\ref{lem:density} and
Proposition~\ref{prop:threefaces} is carried by the development's
section-ultrafilter lemmas, packaged differently there, and the
description of $T$ as an ultraproduct in
Proposition~\ref{prop:threefaces} is not formalized; the final
strict-containment assertion of Proposition~\ref{prop:sparse-not-fiber}
is proved only in the text, the formalized statement being the
existential one; and Example~\ref{ex:fubini}, like the assertions
recorded without proof in Remark~\ref{rem:baumgartner}, lies outside
the formalization.  Two
observations of independent interest emerged from the formalization.
First, the maximality of $\fF(S)$ in Theorem~\ref{thm:ultrapower}
admits a proof different from the one given above: extend $\fF(S)$ by
Zorn's Lemma to an \upf{} $\mathfrak{G}$, identify $h(\mathfrak{G}) = \fu$ directly, and
deduce from the maximality of $S(\mathfrak{G})$ together with the two inversion
identities that $\mathfrak{G} = \fF(S)$.  Second, the conclusion of
Proposition~\ref{prop:sparse-not-fiber} requires only the
isolated-point count; the strict containment
$\overline S \subsetneq K_V$ established in its proof is not needed for
the stated result, though it retains independent interest as exhibiting
sparse proper subsets of fibers.

\subsection*{Acknowledgment}
We thank James Baumgartner for the observation recorded in
Remark~\ref{rem:baumgartner}.


\begin{thebibliography}{99}
\bibitem{booth}
  D.\ Booth,
  \textit{Ultrafilters on a countable set}.
  Ann.\ Math.\ Logic \textbf{2} (1970/71), 1--24.

\bibitem{chang-keisler}
  C.\,C.\ Chang and H.\,J.\ Keisler,
  \textit{Model Theory}, 3rd ed.
  North-Holland, Amsterdam, 1990.

\bibitem{comfort-negrepontis}
  W.\,W.\ Comfort and S.\ Negrepontis,
  \textit{The Theory of Ultrafilters}.
  Springer-Verlag, Berlin, 1974.

\bibitem{dejager-kunzi}
  E.\,P.\ de Jager and H.-P.\,A.\ K\"unzi,
  \textit{Atoms, anti-atoms and complements in the lattice of
  quasi-uniformities}.
  Topology Appl.\ \textbf{153} (2006), 3140--3156.

\bibitem{fms}
  T.\ Frayne, A.\,C.\ Morel and D.\,S.\ Scott,
  \textit{Reduced direct products}.
  Fund.\ Math.\ \textbf{51} (1962), 195--228.

\bibitem{halbeisen-loewe}
  L.\ Halbeisen and B.\ L\"owe,
  \textit{Ultrafilter spaces on the semilattice of partitions}.
  Topology Appl.\ \textbf{115} (2001), 317--332.

\bibitem{jech}
  T.\ Jech,
  \textit{Set Theory}, the third millennium edition.
  Springer-Verlag, Berlin, 2003.

\bibitem{keisler}
  H.\,J.\ Keisler,
  \textit{A survey of ultraproducts}, in: Y.\ Bar-Hillel (ed.),
  Logic, Methodology and Philosophy of Science.
  North-Holland, Amsterdam, 1965, 112--126.

\bibitem{matet}
  P.\ Matet,
  \textit{Partitions and filters}.
  J.\ Symbolic Logic \textbf{51} (1986), 12--21.

\bibitem{pelant-reiterman}
  J.\ Pelant and J.\ Reiterman,
  \textit{Atoms in uniformities}.
  Seminar Uniform Spaces 1973/74, Math.\ Inst.\ Czechoslovak Acad.\
  Sci., Prague, 1975, 73--81.

\bibitem{prrs1}
  J.\ Pelant, J.\ Reiterman, V.\ R\"odl and P.\ Simon,
  \textit{Ultrafilters on $\omega$ and atoms in the lattice of
  uniformities I}.
  Topology Appl.\ \textbf{30} (1988), 1--17.

\bibitem{prrs2}
  J.\ Pelant, J.\ Reiterman, V.\ R\"odl and P.\ Simon,
  \textit{Ultrafilters on $\omega$ and atoms in the lattice of
  uniformities II}.
  Topology Appl.\ \textbf{30} (1988), 107--125.

\bibitem{walker}
  R.\ C.\ Walker,
  \textit{The Stone--\v{C}ech Compactification}.
  Springer-Verlag, Berlin, 1974.
\end{thebibliography}
\end{document}